\documentclass[11pt]{amsart}
\usepackage{xypic}
\usepackage{amsmath, amssymb}

\newcommand{\grp}[6]
{\quad\diagram
\bullet
\xline[#1]
\xline[#2]
 \xline[#3]
& 
\bullet
\xline[#4]
\xline[#5]
\\
\bullet
\xline[#6]
 & 
\bullet
\enddiagram\quad}

\newcommand{\bfs}{\mathbf s}
\newcommand{\bft}{\mathbf t}

\newcommand{\bbN}{{\mathbb N}}

\newcommand{\bbA}{{\mathbb A}}

\newcommand{\bbC}{\mathbb C}
\newcommand{\bbT}{\mathbb T}

\newcommand{\cZ}{{\mathcal Z}}

\newcommand{\rs}{\restriction}

\newcommand{\cB}{\mathcal B}

\newcommand{\e}{\varepsilon}
\newtheorem{thm}{Theorem}[section]
\newtheorem{theorem}{Theorem}
\newtheorem{corollary}[thm]{Corollary}
\newtheorem{coro}[theorem]{Corollary}

\newtheorem{conjecture}[thm]{Conjecture}

\newtheorem{question}[thm]{Question}

\newtheorem{lemma}[thm]{Lemma}

\theoremstyle{definition}

\newtheorem{definition}[thm]{Definition}

\newcounter{my_enumerate_counter}
\newcommand{\pushcounter}{\setcounter{my_enumerate_counter}{\value{enumi}}}
\newcommand{\popcounter}{\setcounter{enumi}{\value{my_enumerate_counter}}}

\def\rs{\restriction}

\newcommand{\FileName}[1]{\thanks{Filename: {\tt #1}}}

\newcommand{\dirlim}{\underrightarrow{\lim}}

\title{Graphs and  CCR algebras}

\author{Ilijas Farah}

\address{Department of Mathematics and Statistics\\
York University\\
4700 Keele Street\\
North York, Ontario\\ Canada, M3J 1P3\\
and Matematicki Institut, Kneza Mihaila 34, Belgrade, Serbia}

\urladdr{http://www.math.yorku.ca/$\sim$ifarah}

\email{ifarah@mathstat.yorku.ca}

\thanks{Partially supported by NSERC}
\FileName{2009f04-nonhomogeneous.tex}

\date{\today.}
\begin{document}
\begin{abstract} 
I introduce yet another way to associate a C*-algebra to a graph
and construct a simple nuclear C*-algebra that has irreducible representations both on 
 a separable and a nonseparable Hilbert space. 
\end{abstract} 

\maketitle

Kishimoto, Ozawa and Sakai have proved in \cite{KiOzSa} that the pure
state space of every separable simple C*-algebra is
\emph{homogeneous} in the sense that for every two pure states $\phi$
and $\psi$ there is an automorphism $\alpha$ such that $\phi\circ
\alpha=\psi$. They have shown that this fails for nonseparable
algebras and asked whether the pure state space of every nuclear (not
necessarily separable) C*-algebra is homogeneous.

\begin{theorem} \label{T3} There is a simple nuclear C*-algebra $B$ 
that has  irreducible representations both on a separable Hilbert space and
on a nonseparable Hilbert space. 
\end{theorem}

\begin{coro} 
There is a simple nuclear algebra whose pure state space 
is not homogeneous. This algebra moreover has a faithful representation 
on  a separable Hilbert space. \qed
\end{coro}

As a curious side result, our construction  gives a non-obvious equivalence relation  
on the class of all graphs. 
For example, among the graphs with four vertices there are three equivalence classes: 
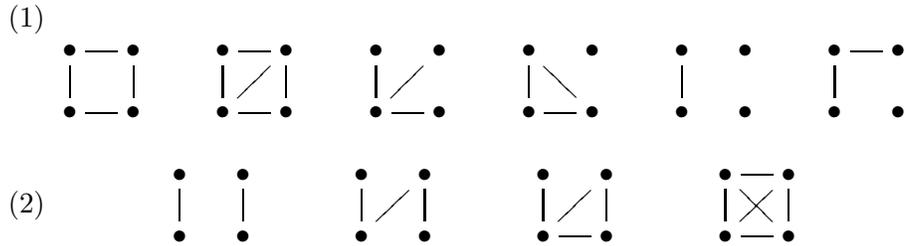
\spreaddiagramcolumns{-1pc}
\spreaddiagramrows{-1pc}
\begin{equation*}
\grp dr{}d{}r
\grp dr{}{dl}dr 
\grp d{}{}{dl}{}r 
\grp d{dr}{}{}{}r 
\grp d{}{}{}{}{}
\grp dr{}{}{}{}\tag {1}
\end{equation*}
\begin{equation*}
\grp d{}{}d{}{}\quad
\grp d{}{}{dl}d{}\quad
\grp d{}{}{dl}dr\quad
\grp dr{dr}{dl}dr \tag{2}
\end{equation*}
and the third one containing the null graph. 
I don't know whether there is a simple description of this  relation 
or what is its computational complexity (see Question~\ref{Q.Complexity}).

In \S\ref{S1}  we prove Theorem~\ref{T3} and in 
\S\ref{S2} we study some properties of the canonical commutation relation (CCR) algebras 
associated with graphs of which the algebra used in the proof of Theorem~\ref{T3} is a special case. 
By $|X|$ we denote the cardinality of the set $X$. 
All C*-algebras considered in this paper will be nuclear and therefore 
the notation $A\otimes B$ will always be unambiguous. 
We also use the following convention. 
If $A$ and $B$ are unital algebras then $A\otimes B$ is identified with  a subalgebra of $B$. 
Similarly, if $A_i$, for $i\in X$, are unital algebras and $Y\subseteq X$ then
$\bigotimes_{i\in Y} A_i$ is considered as a subalgebra of $\bigotimes_{i\in X} A_i$. 
Note that under our assumptions this makes sense for arbitrary sets $X$ and $Y$. 
 All the background can be found in \cite{Black:Operator} and~\cite{We:Set}.

\section{Graphs and algebras}

\label{S1}

Given a graph $G=(V,E)$ let $B(G)$ be the universal algebra generated by unitaries 
$u_x$, for $x\in V$ that satisfy relations
\begin{align*} 
u_x u_x^*&=1 &\text{ for all $x$,}\\
u_x^2&=1&\text{ for all $x$,}\\
u_x u_y&=u_y u_x& \text{ if $x$  and $y$ are not adjacent,}\\
u_x u_y&=-u_y u_x& \text{ if $x$  and $y$ are  adjacent.}
\end{align*}
Recall that the \emph{character density} of a $C^*$ algebra is the minimal cardinality 
of its  dense subset. 

\begin{lemma}\label{L.well-defined} The algebra $B(G)$ is well-defined for every graph $G$,
and its character density is equal to $|G|+\aleph_0$. 
\end{lemma}

\begin{proof} We first show that for every finite graph 
$G$ there is a C*-algebra generated by the unitaries $u_x$, for $x\in V$, satisfying the 
required relations. 

Let $n=|V|$ and let $m=\binom n 2$, identified with the set of distinct pairs $\{i,j\}$ 
of natural numbers in $\{1,\dots, n\}$. For each pair  $1\leq i<j\leq n$ fix a two-dimensional 
complex Hilbert space $H_{ij}$ and let $H=\bigotimes_{1\leq i<j\leq n} H_{i,j}$. 

For $k\leq n$ define the unitary $u_k$ on $H$ as
\[
u_k=\bigotimes_{1\leq i<j\leq n} u_{i,j,k}
\]
where 
\[
u_{i,j,k}=
\begin{cases} 
\begin{pmatrix} 1 & 0 \\ 0 & 1 \end{pmatrix} & 
\text{ if $k\notin \{i,j\}$ or $i$ is not adjacent to $j$}\\
\begin{pmatrix} 1 & 0 \\ 0 & -1 \end{pmatrix} & 
\text{ if $k=i$ and $i$ is adjacent to $j$, and}\\
\begin{pmatrix} 0 & 1 \\ 1 & 0 \end{pmatrix} & 
\text{ if $k=j$ and $i$ is adjacent to $j$.}
\end{cases} 
\]
Then each $u_k$ is a self-adjoint unitary and 
clearly $u_i$ and $u_j$ commute if~$i$ is not adjancent to $j$ and $u_i$ and $u_j$ anti-commute if 
 $i$ is adjacent to $j$. 
 Therefore $C^*(\{u_i\colon i\leq n\})$ realizes the defining relations for $B(G)$. 
If $G$ is infinite, then clearly $B(G)$ is the direct limit of $B(G_0)$ where $G_0$ ranges
over all finite subgraphs of $G$. 
Therefore for every $G$ there is a C*-algebra that realizes the defining relations for $G$. 

Since all the generators of $B(G)$ are unitaries, by taking the direct sum of all representations
one obtains $B(G)$ for a finite $G$. 

We claim that $x\neq y$ implies $\|u_x-u_y\|\geq \sqrt 2$. Since the matrices
$\begin{pmatrix} 1 & 0 \\ 0 & 1 \end{pmatrix}$, $\begin{pmatrix} 1 & 0 \\ 0 & -1 \end{pmatrix}$ and
$\begin{pmatrix} 0 & 1 \\ 1 & 0 \end{pmatrix}$ are $\sqrt 2$ apart from each other, this will 
follow from Lemma~\ref{L.sqrt2}. 
Since the generating unitaries $u_i$, for $i\in V$, form a discrete generating set, 
the character density of $B(G)$ is $|G|$ if $G$ is infinite and $\aleph_0$ if $G$ is finite.  
\end{proof} 

The following lemma is probably  
well-known but I could not find a reference (here 
$\bbT$ denotes the unit circle in $\bbC$). 

\begin{lemma} \label{L.sqrt2}  In any spatial tensor product 
of C*-algebras $C\otimes D$ the following holds.  
If $v$ and $w$ are unitaries in $D$ and $a$ and $b$ are in $C$ then 
\[
\|a\otimes v -b\otimes w\|\geq \inf_{\lambda \in \bbT} \|\lambda a-b\|.
\]
\end{lemma} 

\begin{proof} Fix a representation of $C\otimes D$ on $H_1\otimes H_2$. 
Fix $\e>0$ 
and $\lambda$ in the spectrum of $w^*v$.  
Pick a unit vector $\eta$ in $H_2$ such   that $\|w^*v\eta-\lambda\eta\|<\e$. 
Now find a unit vector $\xi$ in $H_1$ such that $\|(\lambda a-b)\xi\|>\|a-b\|-\e$. 
Then 
\begin{align*}
\|(a\otimes v-b\otimes w)(\xi\otimes \eta)\|
&= \|(\lambda a\otimes v -b\otimes \lambda w)(\xi\otimes \eta)\|\\
&\geq \|((\lambda a-b)\otimes v)(\xi\otimes \eta)\|-\|(b\otimes (v-\lambda w))(\xi\otimes \eta)\|\\
&>\|\lambda a-b\|-\e(1+\|b\|). 
\end{align*}
Since $\e>0$ was arbitrary, the conclusion follows. 
\end{proof}

The algebra  in \eqref{L.M-n.2} of Lemma~\ref{L.M-n} below, 
with $n=4$, corresponds to 
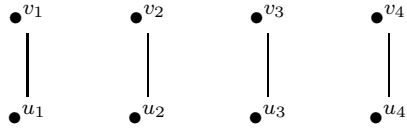
\spreaddiagramcolumns{+1pc}
\spreaddiagramrows{+1pc}
\[
\diagram
\bullet^{v_1}\xline[d] &
\bullet^{v_2}\xline[d] &
\bullet^{v_3}\xline[d] &
\bullet^{v_4}\xline[d] \\
\bullet^{u_1}&
\bullet^{u_2}&
\bullet^{u_3}&
\bullet^{u_4}
\enddiagram
\]
and the algebra in \eqref{L.M-n.3} of the same lemma, with $l=2$ and $n=2$, corresponds 
to any graph of the form (the dashed line means that the vertices may or may not be adjacent)
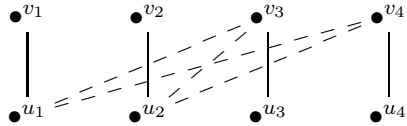
\[
\diagram
\bullet^{v_1}\xline[d] &
\bullet^{v_2}\xline[d] &
\bullet^{v_3}\xline[d]\xdashed[lld]\xdashed[ld] &
\bullet^{v_4}\xline[d]\xdashed[llld]\xdashed[lld] \\
\bullet^{u_1}&
\bullet^{u_2}&
\bullet^{u_3}&
\bullet^{u_4}
\enddiagram
\]
The proof of  Lemma~\ref{L.M-n} is implicit in \cite{FaKa:Nonseparable} but we sketch it for the 
reader's convenience. A related result is proved in Lemma~\ref{L.graphs} below.

\begin{lemma} \label{L.M-n} For a  C*-algebra $A$ the following
are equivalent. 
\begin{enumerate}
\item \label{L.M-n.1}  $A$ is isomorphic to $M_{2^n}(\bbC)$.
\item\label{L.M-n.2}    $A$ is generated by self-adjoint unitaries $u_1,\dots, u_n$ and $v_1,\dots, v_n$ such that 
$u_i$ and $v_j$ commute if and only if $i=j$
and $u_i$ an $v_j$ anti-commute if and only if $i\neq j$. 
\item\label{L.M-n.3} 
   $A$ is generated by self-adjoint unitaries $u_1,\dots, u_n$ and $v_1,\dots, v_n$ such that 
for some $l\leq n$ we have 
 \begin{enumerate}
 \item If $j\leq l$ then $u_i$ and $v_j$ anti-commute if and only if~$i=j$.
  \item If $l<i$ then $u_i$ and $v_j$ anti-commute if and only if $i=j$ 
\end{enumerate}
\end{enumerate}
\end{lemma} 

\begin{proof}  
The case $n=1$ is \cite[Lemma~4.1]{FaKa:Nonseparable}, 
using  
\[
u_1=\left(\begin{matrix} 1 & 0 \\ 0 & -1 \end{matrix}\right)
\qquad \text{and} \qquad
v_1=\left(\begin{matrix} 0 & 1 \\ 1 & 0 \end{matrix}\right)
\]
and the fact that $A$ is a noncommutative C*-algebra that is a 4-dimensional vector
space over $\bbC$ for the converse. 

Fix $n>1$. 
Note that $M_{2^n}(\bbC)$ is 
isomorphic to $\bigotimes_{i=1}^n M_2(\bbC)$. 
Using the convention stated before the lemma, identify 
the unitaries $u_i$ and $v_i$ generating the $i$-th  copy of $M_2(\bbC)$ 
with elements of $M_{2^n}(\bbC)$. Then $u_i, v_i$, for $1\leq i\leq n$ are 
as in \eqref{L.M-n.2}

To see that \eqref{L.M-n.2} implies \eqref{L.M-n.1}, 
assume $A$ is generated by $u_i,v_i$, for $1\leq i\leq n$, 
as in the statement of the lemma. Then $A_i=C^*(u_i,v_i)$ is a subalgebra of $A$
isomorphic to $M_2(\bbC)$. These subalgebras are commuting and they 
generate~$A$, and therefore \eqref{L.M-n.1} follows. 

Since \eqref{L.M-n.3} is a special case of \eqref{L.M-n.2} (with either $l=1$ or $l=n$)
it remains to prove \eqref{L.M-n.3} implies \eqref{L.M-n.2}. 
For $l<j\leq n$ define 
 \begin{align*}
 K(j)&=\{i\leq l: v_j u_i=-u_i v_j\}.
 \end{align*}
 For all $m\leq n$ we have that 
 $w_j=v_j\prod_{i\in K(j)} v_i$
 commutes with $u_m$ if $m\neq j$ and anticommutes with $u_m$ if $m=j$.
  
Let $w_j=v_j$ for $j\leq l$. 
Since for $l<j\leq n$ we have $v_j=w_j\prod_{i\in K(j)} w_i$, $A$ is generated by 
$w_1,\dots, w_n$ and $u_1,\dots, u_n$ and they satisfy \eqref{L.M-n.2}. 
\end{proof}

 For a set $Y$ 
identify the power-set of $Y$ with $2^Y$ and consider it with the product topology. 
If $\bbA\subseteq 2^Y$ then let $G(Y,\bbA)$ denote the bipartite graph 
with the set of vertices $Y\cup \bbA$ such that $i\in Y$ and $x\in \bbA$ are adjacent
if and only if $i\in x$. 

\begin{lemma} \label{L.B} 
Assume $\bbA\subseteq 2^Y$. Then   
 the C*-algebra $B=B(G(Y,\bbA))$ has a representation on a Hilbert space of density $|Y|$. 
  If $\bbA$ is dense in $2^Y$ then this representation 
 can be chosen to be
irreducible. 
\end{lemma} 

\begin{proof}  
We shall denote the generating untaries by $u_i$, $i\in Y$ and $v_x$, $x\in \bbA$. 

For each pair $i\in Y$, $x\in \bbA$ let $H_{i,x}$ be the two-dimensional 
complex Hilbert space and let $\zeta_{i,x}$ denote the 
vector $\begin{pmatrix} 1\\ 0\end{pmatrix}$ in $H_{i,x}$. 
We shall represent $B$ on $H=\bigotimes_{i\in Y} (H_{i},\zeta_{i})$. 
(Recall that this is the closure of the linear span of elementary tensors  of the form 
$\bigotimes_{i} \xi_{i}$ such that $\xi_{i}=\zeta_{i}$ for all but finitely many pairs
$i$.) Since the character density of $H$ is equal to $|Y|$, this will prove the claim. 
For $i\in Y$ let $u_i\in\cB(H)$ be defined by 
\[
\textstyle u_i=\bigotimes_{j\in Y} u_{ij}
\]
where $u_{ii}=
\begin{pmatrix} 0 & 1 \\ 1 & 0 \end{pmatrix} $ and $u_{ij}$ is the identity matrix whenever $i\neq j$. 
For $x\in \bbA$ let $v_x\in \cB(H)$ be defined by 
(using  the convention that  the omitted terms are equal to the identity matrix)
\[
v_x=\bigotimes_{i\in x} \begin{pmatrix} 1 & 0 \\ 0 & -1\end{pmatrix}. 
\]
Since 
$\textstyle\begin{pmatrix} 1 & 0 \\ 0 & -1\end{pmatrix} \begin{pmatrix} 1\\ 0\end{pmatrix} =
\begin{pmatrix} 1\\ 0\end{pmatrix}$, every elementary tensor of the form 
$\bigotimes_i \xi_i$ such that $\xi_i=\zeta_i$ for all but finitely many $i$ is sent to an 
elementary tensor of this form. 
Since $H$ is the closed linear span of such vectors, 
 $v_x$ is an operator on~$H$.

 It is clear that $v_x$ and $u_i$ commute if $i\notin x$ and that 
 $v_x$ and $u_i$ anticommute if $i\in x$. Since $B$ was assumed to be simple, 
 it is isomorphic to the 
  algebra $C^*(\{u_i\colon i\in Y\}\cup \{v_x\colon x\in \bbA\})$. 
  
 Now assume $\bbA$ is dense in $2^Y$. 
For  $F\subseteq Y$ write $H_F$ for $\bigotimes_{i\in F} H_i$. 
Fix a finite $F\subseteq Y$ and write  
\[
\textstyle\zeta=\bigotimes_{i\in Y\setminus F} \zeta_i.
\] 
Therefore $\xi\in H_F$ implies $\xi\otimes\zeta\in H_Y$. 
For every $x\subseteq Y$ and every $\xi\in H_F$ we have $v_x(\xi\otimes \zeta)=
(v_{x\cap F}\xi)\otimes\zeta$. 
Since $\bbA$ is dense in $2^Y$,   Lemma~\ref{L.M-n} 
implies $C^*(\{u_i: i\in F\}\cup \{v_{x\cap F}: x\in \bbA\})=\cB(H_F)$. 
Therefore for any two unit vectors  $\xi\otimes \zeta$  and $\eta\otimes\zeta$
there is $a\in B$ such that $a\xi=\eta$. Since $H_Y$ is the 
direct limit of $H_F$ for $F\subseteq Y$ finite, we conclude that $H_Y$ has 
no nontrivial closed $B$-invariant subspace. 
\end{proof}

\begin{definition}\label{D.independent}
A family $\bbA$ of subsets of $Y$ is \emph{independent} if for all finite disjoint subsets 
$F\neq \emptyset$ and $G$ of $\bbA$ we have that $\bigcap F\setminus \bigcup G$ is nonempty. 
\end{definition} 
It is not difficult to see  that if $\bbA$ is infinite then 
 this is equivalent to requiring such intersections to  always be infinite. 
The proof of this  fact is included in the proof of 
Lemma~\ref{L.hat}.

A \emph{full matrix algebra} is an algebra of the form $M_n(\bbC)$. Following \cite{FaKa:Nonseparable} 
we say that an algebra is AM (approximately matricial) if it is a direct limit of full matrix algebras. 
The following lemma will be generalized in Lemma~\ref{L.characterization}. 

\begin{lemma}  \label{L.directed} 
Assume $\bbA$ is infinite, independent, and dense in $2^Y$. 
Then $B=B(G(Y,\bbA))$  is simple, nuclear, unital  and it has 
the unique trace. 
\end{lemma}

\begin{proof} 
We shall denote the generating untaries by $u_i$, $i\in Y$ and $v_x$, $x\in \bbA$. 
It suffices to prove  that $B$ is AM, since  every AM algebra is   
simple, nuclear, unital, and has the unique trace (\cite{FaKa:Nonseparable}). 

 Let $\Lambda_0$ be the set of all pairs $(F,G)$ such
 that $F\subseteq Y$ is finite, $G\subseteq \bbA$ is finite ordered by 
 the coordinatewise inclusion. 
 With 
 \[
 D(F,G)=C^*(\{u_i: i\in F\}\cup \{v_x: X\in G\})
 \]
 we have that $B=\dirlim_{\Lambda_0} D(F,G)$. 
 Now let $\Lambda$ be the set of all $(F,G)\in \Lambda_0$ such that 
 for some 
  $1\leq l \leq n\in \bbN$  we have the following. 
 \begin{enumerate}
 \item $F=\{x(1),\dots, x(n)\}$ and $G=\{k(1),\dots, k(n)\}$, 
 \item If $j\leq l$ then  $k(i)\in x(j)$ if and only if $i=j$,
 \item If $l<i$ then  $k(i)\in x(j)$ if and only if $i=j$.
 \end{enumerate}
Lemma~\ref{L.M-n}  implies that $D(F,G)$ is isomorphic to $M_{2^n}(\bbC)$ (with $n$ as above) and
it therefore suffices to prove that 
 $\Lambda$ is cofinal in $\Lambda_0$. 

Fix $(F,G)\in \Lambda_0$. We may assume $|F|=|G|=l$ and
enumerate them as $F=\{x(i): l<i\leq 2l\}$ and $G=\{k(i): i\leq l\}$. 
Since $\bbA$ is independent, for each $j$ such that $l<j\leq 2l$ we can pick 
\[
\textstyle k(j)\in x(j)\setminus \left(\bigcup_{l<i\leq 2l} x(i)\cup G\right).
\]
By the density of $\bbA$ for each $j\leq l$ pick $x(j)\in \bbA$ such that 
\[
x(j)\cap \{k(1),\dots, k(2l)\}=\{k(j)\}. 
\]
Let $F'=\{x(1),\dots, x(2l)\}$ and $G'=\{k(1),\dots, k(2l)\}$ and $n=2l$ we see that 
$(F',G')$ is in $\Lambda$, concluding the proof. 
\end{proof}

For a family $\bbA$ of subsets of $Y$ consider the dual family $\hat\bbA=\{z(i): i\in Y\}$ of subsets
of $\bbA$ defined by 
\[
x\in z(i)\text{ if and only if } i\in x.
\]
In the following lemma we   identify $\hat\bbA$ and $Y$ by identifying $i\in Y$ with $z(i)\in \hat\bbA$.

\begin{lemma}\label{L.hat} Assume  $Y,\bbA$, and $\hat \bbA$ are as above. 
\begin{enumerate}
\item  
the dual, $\hat{\hat \bbA}$, of $\hat\bbA$ is equal to  $\bbA$. 
\item $\bbA$ is dense in $2^Y$ if and only if $\hat\bbA$ is independent. 
\item $\hat\bbA$ is dense in $2^{\bbA}$ if and only if $\bbA$ is independent. 
\end{enumerate}
\end{lemma}

\begin{proof} The first assertion is obvious and the third follows immediately from the first two. 

Assume $\bbA$ is not dense in $2^Y$ and fix a nonempty 
basic open set $U\subseteq 2^Y$ disjoint from $\bbA$. 
For some finite and disjoint $F\subseteq Y$ and $G\subseteq Y$ we have
that $U=\{x\in 2^Y: x\cap F=\emptyset$ and $G\subseteq x\}$. 
The Boolean combination   $\bigcap F\setminus \bigcup G=\emptyset$ 
witnesses that $\hat\bbA$ (identified with $Y$)
is not independent. 
Now assume $\bbA$ is dense in $2^Y$. This implies that its intersection 
with every nonempty basic open set is nonempty (and moreover 
infinite if $Y$ is infinite), and by the above argument $\hat\bbA$
is independent. 
\end{proof}

 We include a proof of the following classical result (\cite[(A6) on p. 288]{Ku:Book})
 for reader's convenience. 
 
\begin{lemma} [Fichtenholz--Kantorovich]
\label{L.independent} 
There exists an independent family of subsets of $\bbN$ of cardinality continuum. 
\end{lemma} 

\begin{proof} 
Let $2^{<\bbN}$ denote the set of all finite sequences of $0$s and $1$s. 
Our family will consist of subsets of the (countable) set of all finite subsets of $2^{<\bbN}$. 
For $f\in 2^{\bbN}$ by $f\rs m$ we denote its initial segment of length $m$ and for $s\in 2^{<\bbN}$ by 
$|s|$ we denote its length. 
For $f\in 2^{\bbN}$ let 
\begin{equation*}
X_f=\{T\subseteq 2^{<\bbN}\colon 
(\exists m) |s|=m\text{ for all $s\in T$ and $f\rs m\in T$}\}.
\end{equation*}
Assume $m<n$ and $f_1,\dots, f_m, f_{m+1}, \dots f_n$ are distinct elements of $2^{\bbN}$. 
Fix $k$ large enough so that $f_i\rs k\neq f_j\rs k$ for all $i\neq j$. 
Then 
\begin{equation*}
s=\{f_i\rs k: i\leq m\}
\end{equation*}
belongs to $\bigcup_{i=1}^m X_{f_i} \setminus \bigcup_{j=m+1}^n X_{f_j}$. 
Therefore the family $\{X_f: f\in 2^{\bbN}\}$ is independent. 
\end{proof}

\begin{proof}[Proof of Theorem~\ref{T3}]
Let $\bbA$ be an independent family of subsets of $\bbN$ of 
size continuum  as in Lemma~\ref{L.independent}. 
The remark after Definition~\ref{D.independent} implies that  
 if $x\in \bbA$ is replaced with $x'$ such that the symmetric difference 
 $x\Delta x'$ is finite, then $(\bbA\cup\{x'\})\setminus \{x\}$ is still independent. 
 By making finite changes to countably many of the members of $\bbA$ we can therefore 
 assure $\bbA$ is both  
 dense and independent in $2^{\bbN}$. 
 By Lemma~\ref{L.B} the C*-algebra $B=B(G(\bbN,\bbA))$ 
  has an irreducible representation on a separable Hilbert space. 
Since the graphs $G(Y,\bbA)$ and $G(\bbA,\hat\bbA)$ are isomorphic, 
 $B$ is isomorphic to $B(G(\bbA,\hat\bbA))$. Since $|\bbA|=2^{\aleph_0}$, 
Lemma~\ref{L.B} implies that $B$ has an irreducible representation on a nonseparable Hilbert space. 
\end{proof}

The assumptions  of Lemma~\ref{L.directed} can be weakened. 
Instead of requiring~$\bbA$ to be independent, we may require 
  that 
for every $x\in \bbA$ and every finite $F\subseteq \bbA\setminus \{x\}$ 
the set $x\setminus \bigcup F$ is nonempty. 
Instead of requiring $\bbA$ to be dense, we can require 
that for every finite $s\subseteq Y$
and every $j\in s$ there is $x\in \bbA$ such that $x\cap s=\{j\}$. 
The proof of Lemma~\ref{L.hat} shows that $\bbA$ satisfies these two conditions
if and only if $\hat\bbA$ satisfies these two condtions. Therefore instead of an independent 
family, in the proof of Theorem~\ref{T3} 
 we could have used an \emph{almost disjoint} family, i.e., a family $\bbA$ of infinite subsets of $Y$ 
 such that $x\cap y$ is finite for all distinct $x$ and $y$ in $\bbA$. 
Uncountable almost disjoint families in $2^{\bbN}$ are well-studied set-theoretic objects. 

\section{More on algebras and graphs}
\label{S2}
 
Note that if $|V|=n$ then $B(G)$ is a $2^n$-dimensional vector space over~$\bbC$ 
since it is spanned by $v_{\bfs}=\prod_{x\in \bfs} v_x$ for $\bfs\subseteq V$ ($v_{\bfs}$ are defined
using a fixed linear order on $V$ for definiteness). 
On the collection of all graphs define the equivalence relation $\sim$ by 
$G_1\sim G_2$ if $B(G_1)$ and $B(G_2)$ are isomorphic. 

For a graph $G=(V,E)$, a finite subset $\bfs$ of $V$ and $x\in \bfs$ define
the graph $G-x+\bfs$ as follows. It vertex set is $V'=V\setminus \{x\}\cup \{\bfs\}$, 
hence $\bfs$ is considered as a vertex in the new graph. 
The adjacency relation  for vertices in $V\setminus\{x\}$ is unchanged, and 
we let $\bfs$ be adjacent to $u\in V\setminus \{x\}$ if and only if 
 $|\{w\in\bfs\colon \{w,u\}\in E\}|$ is an odd number.

 \begin{lemma} \label{L.switch} For $G,x$ and $\bfs$ as above the algebras $B(G)$ and 
 $B(G-x+\bfs)$ are isomorphic. 
 \end{lemma} 

\begin{proof} In $B(G)$ 
 consider 
the product
$
u_{\bfs}=\prod_{i\in \bfs} u_i
$
(for definiteness, we are assuming that $V$ is well-ordered and the unitaries in the product are taken 
in this order). 
Then $u_{\bfs}$ is a unitary and one of  $u_{\bfs}$ and $i u_{\bfs}$ is self-adjoint, 
depending on whether the number of edges between the vertices in  $\bfs$  is even or odd. 
Let $w_{\bfs}$ denote this self-adjoint unitary.
Then the unitaries $\{u_x\colon x\in V\setminus \{x\}\}\cup\{w_{\bfs}\}$ clearly satisfy the 
relations corresponding to $G-x+\bfs$. 

Since $x\in\bfs$, in $B(G-x+\bfs)$ we can similarly define a unitary $w_x$ 
such that the unitaries 
$\{u_x\colon x\in V\setminus \{x\}\}\cup\{w_x\}$ satisfy the 
relations corresponding to $G$. 

We have shown that 
every algebra generated by unitaries satisfying relations corresponding 
to $G$ is also generated by unitaries satisfying relations corresponding to 
$G-x+\bfs$, and vice versa. Since this correspondence is given in a canonical 
way, we conclude that the universal algebras are isomorphic. 
\end{proof}

\begin{lemma} \label{L.graphs} 
For every graph $G$, if $|G|=n$ then there is $k\leq n/2$ such that with $l=n-2k$
we have that $B(G)$ is isomorphic to $M_{2^k}(\bbC)\otimes \bbC^{2^l}$. 
\end{lemma} 

\begin{proof}  We need to show that every graph $G$ with $n$ vertices is equivalent to 
a graph of the form
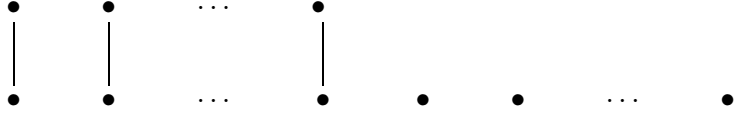
\begin{equation*} 
\diagram
\bullet\xline[d]& \bullet\xline[d]& \dots & \bullet\xline[d]\ &  &  &  &  & \\
\bullet& \bullet& \dots & \bullet 
& \bullet& \bullet & \dots &\bullet 
\enddiagram
\end{equation*}
where there are $k$ pairs of vertices on the left hand side and $l=n-2k$ vertices on the right hand side. 
We shall refer to this graph as `the canonical graph representing $M_{2^k}(\bbC)\otimes \bbC^{2^l}$.'

The proof is by induction on $n$. If $n=1$ or $n=2$ 
 then the assertion is vacuous. 
 We shall first prove the case $n=3$ both as a warmup and because  it will be used in the inductive step. 
 We shall  prove that each graph $G$ on three vertices is isomorphic either 
 to the null graph or to the graph with a single edge. By using Lemma~\ref{L.switch} we have
 the following. 
   
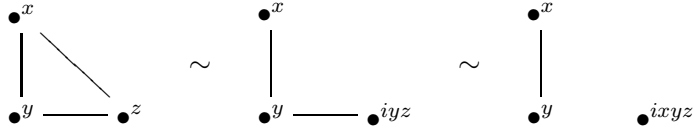
\begin{equation*}
\diagram 
\bullet^x \xline[d]\xline[dr] \\
\bullet^y \xline[r] & \bullet^z
\enddiagram
\quad\sim\quad 
\diagram 
\bullet^x \xline[d] \\
\bullet^y \xline[r] & \bullet^{iyz}
\enddiagram
\quad\sim\quad
\diagram 
\bullet^x \xline[d] \\
\bullet^y & \bullet^{ixyz}
\enddiagram
\end{equation*}
Since $G_1$, $G_2$ and $G_3$, together with the null graph, are all graphs with three vertices, 
this concludes the proof of the case $n=3$. 

Assume the assertion is true for $n$ and fix $G$ such that $|V|=n+1$. 
Applying the inductive hypothesis, we may assume that the induced graph  of $G$ to the first
$n$ vertices is the canonical graph representing $M_{2^k}(\bbC)\otimes \bbC^{2^{2l}}$ 
for some $k$ and $l$. 
Then $G$ is of the form 
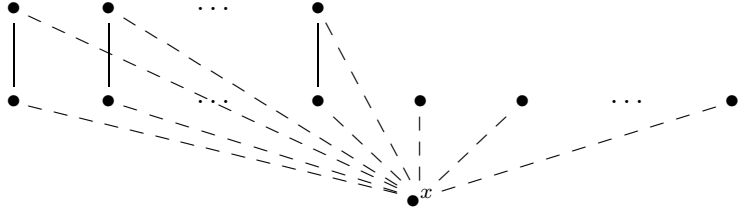
\begin{equation*} 
\diagram
\bullet\xline[d]\xdashed[ddrrrr] & \bullet\xline[d]\xdashed[ddrrr] & \dots & \bullet\xline[d]\xdashed[ddr] &  &  &  &  & \\
\bullet\xdashed[drrrr]& \bullet\xdashed[drrr] & \dots & \bullet\xdashed[dr] 
& \bullet\xdashed[d] & \bullet\xdashed[dl] & \dots &\bullet\xdashed[dlll] \\
&&& &   \bullet^x
\enddiagram
\end{equation*}
By the case $n=3$ treated above, each of the triangles on the left hand side of the graph 
can be turned into a graph with exactly one edge (with this edge being the
one not incident with $x$), by multiplying~$x$ with some of the 
other generators and (if necessary) $i$. It therefore remains to check that every graph of the form 
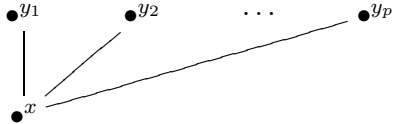
\begin{equation*}
\diagram
\bullet^{y_1}\xline[d] & \bullet^{y_2}\xline[dl] & \dots & \bullet^{y_p} \xline[dlll]\\
\bullet^x
\enddiagram
\end{equation*}
is equivalent to a graph with exactly one edge. This is obtained by replacing $y_j$, for $j\geq 2$, 
with $y_1 y_j$ and using Lemma~\ref{L.switch}. 
\end{proof}

Lemma~\ref{L.graphs} implies there are exactly $1+\lfloor n/2\rfloor$ nonisomorphic algebras
of the form $B(G)$, where $G$ is a graph with $n$ vertices. For example, in the case $n=4$
the algebras are $\bbC^{16}$ (corresponding to the null graph), $M_2(\bbC)\otimes \bbC^{4}$, 
corresponding to any of the graphs
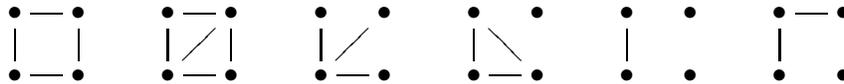
\spreaddiagramcolumns{-1pc}
\spreaddiagramrows{-1pc}
\begin{equation*}
\grp dr{}d{}r
\grp dr{}{dl}dr 
\grp d{}{}{dl}{}r 
\grp d{dr}{}{}{}r 
\grp d{}{}{}{}{}
\grp dr{}{}{}{}
\end{equation*}
and $M_4(\bbC)$, corresponding to any of the graphs
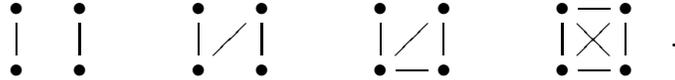
\begin{equation*}
\grp d{}{}d{}{}\quad
\grp d{}{}{dl}d{}\quad
\grp d{}{}{dl}dr\quad
\grp dr{dr}{dl}dr.
\end{equation*}
I don't know whether there is a simpler description of the relation $\sim$, even on the finite graphs, 
then the one given by Lemma~\ref{L.iso} below.

For a graph $G=(V,E)$ 
let $G^{<\infty}$ be the graph whose vertices are all finite nonempty subsets of $V$ 
and two such vertices $\bfs$ and $\bft$ are adjacent if and only if the cardinality of the set
\[
\{(i,j)\colon i\in \bfs, j\in \bft, \text{ and } \{i,j\}\in E\}
\]
is an odd number.

\begin{lemma} \label{L.iso} Assume $G$ and $K$ are graphs. 
\begin{enumerate}
\item\label{L.iso.a}  If graphs $G^{<\infty}$ and $K^{<\infty}$ are isomorphic then 
algebras $B(G)$ and $B(K)$ are isomorphic. 
\item\label{L.iso.b}  Assume $G$ and $K$ are finite. Then the following are equivalent
\begin{enumerate}
\item   \label{L.iso.b.1}  $G^{<\infty}$ and $K^{<\infty}$ are isomorphic, 
\item \label{L.iso.b.2} $G$ can be obtained from $K$ by a finite number of applications of Lemma~\ref{L.switch}, 
  \item   \label{L.iso.b.3} $B(G)$ and 
$B(K)$ are isomorphic. 
\end{enumerate}
\end{enumerate}
\end{lemma} 

\begin{proof} 
\eqref{L.iso.a} Consider the algebra $B(G)$.  
The linear span of unitaries of the form $w_{\bfs}$ (as defined in the proof of Lemma~\ref{L.switch}) 
is dense in $A(G)$. 
If  $G^{<\infty}$ is isomorphic to $K^{<\infty}$ then $A(G)$ and $A(K)$ have 
isomorphic---and therefore isometric---dense *-algebras and  are, therefore,  isomorphic.

It is clear that \eqref{L.iso.b.2} implies \eqref{L.iso.b.1} and 
\eqref{L.iso.b.1} implies \eqref{L.iso.b.3} by part \eqref{L.iso.a}. 

For the remaining implication we need to assume $G$ and $K$ are finite. 
Assume \eqref{L.iso.b.3}.   
 Lemma~\ref{L.graphs} implies that  by a finite number of applications
of Lemma~\ref{L.switch} graph $G$ 
can be turned into the canonical graph representing 
 $M_{2^k}(\bbC)\otimes \bbC^{2^l}$
for some $k$ and $l$. 
Similarly,  by a finite number of applications
of Lemma~\ref{L.switch} graph $K$ can be turned into the canonical graph 
representing  $M_{2^{k'}}(\bbC)\otimes\bbC^{2^{l'}}$ for 
some $k'$ and $l'$. If these algebras are isomorphic then $k=k'$ and $l=l'$, 
and \eqref{L.iso.b.2} follows by transitivity. 
\end{proof} 

Although the equivalence of \eqref{L.c.3} and \eqref{L.c.4} of the following lemma 
is a version of Lemma~\ref{L.directed} in a wider context, the latter is not an immediate
consequence of the former.

\begin{lemma} \label{L.characterization}
For an infinite graph $G=(V,E)$ the following are equivalent. 
\begin{enumerate}
\item\label{L.c.1} 
 The family of finite induced subgraphs $G_0$ of $G$ such that $B(G_0)$ is isomorphic to 
a full matrix algebra is cofinal in all finite induced subgraphs of $G$. 
\item\label{L.c.2}  $B(G)$ is AM. 
\item\label{L.c.3}  $B(G)$ is simple and has a unique trace. 
\item\label{L.c.4}  For all finite nonempty $\bfs\subseteq V$ there is $v\in V$ such that 
$|\{u\in \bfs: u$ is adjacent to $v\}|$ is odd. 
\pushcounter
\end{enumerate}
\end{lemma} 

\begin{proof} The implication from  \eqref{L.c.1} to \eqref{L.c.2} is immediate and 
\eqref{L.c.2} implies \eqref{L.c.3} was proved in \cite{FaKa:Nonseparable}. 
Assume \eqref{L.c.4} fails for some $\bfs$. 
Then the unitary $w_{\bfs}$ (as defined in the proof of Lemma~\ref{L.switch}) 
commutes with every generating unitary $u_x$ of 
 $B(G)$ and therefore belongs to its center, hence
 \eqref{L.c.3} fails.

Now assume \eqref{L.c.4}. We first prove that it is equivalent to 
\begin{enumerate}
\popcounter
\item \label{L.c.5} For all finite nonempty $\bfs\subseteq V$ there is a finite 
$\bft\subseteq V$ such that 
$|\{(u,v)\in \bfs\times \bft: u$ is adjacent to $v\}|$ is odd.
\end{enumerate}
Clearly \eqref{L.c.4} implies \eqref{L.c.5}. The reverse implication holds because if a sum of 
integers is odd then at least one of them has to be odd.

In order to prove \eqref{L.c.1} fix a finite 
induced subgraph $G_0$ of $G$. 
By Lemma~\ref{L.graphs} $G_0$ is $\sim$-equivalent to the canonical 
graph representing  
$M_{2^k}(\bbC)\otimes \bbC^{2^l}$ for some $k$ and $l$.
By Lemma~\ref{L.iso} \eqref{L.iso.b}, we can assume $G_0$ is equal to 
the latter graph (note that the condition \eqref{L.c.5} is invariant under  this change). 

If $l=0$ there is nothing to prove. Otherwise let $x$ be one of the $l$ unmatched vertices
in $G_0$. By \eqref{L.c.4}  pick $y$ in $G$ adjacent to $x$. 
Note that $y$ is not a vertex of $G_0$. 
The  construction of Lemma~\ref{L.graphs} shows that the induced
subgraph of $G$ on $V(G_0)\cup \{y\}$ is equivalent to the canonical 
graph representing 
$M_{2^{k+1}}(\bbC)\otimes \bbC^{2^{l-1}}$.

Repeating this construction $l-1$ more times we find an induced 
 subgraph $G_1$ of $G$ including $G_0$ such that $B(G)$ is isomorphic to 
 $M_{2^{k+l}}(\bbC)$ and \eqref{L.c.1} follows. 
 \end{proof} 

Lemma~\ref{L.characterization}  implies $B(G)$ is 
isomorphic to the CAR algebra $M_{2^\infty}$ for the generic countable graph $G$. 
This is not surprising since the generic countable graph is isomorphic to the Rado graph, also 
known as the random graph. 
Lemma~\ref{L.characterization}  also 
implies that the algebras of the form $B(G)$ do not give new examples of separable
C*-algebras. 

\begin{corollary} If $G$ is a countably infinite
 graph then $B(G)$ is isomorphic to an algebra of the form 
 $M_{2^m}(\bbC)\otimes \bbC^{2^n}$ for $m$ and $n$ in $ \bbN\cup \infty$, 
 where $\bbC^{2^\infty}$ is defined to be $C(2^{\bbN})$, the algebra of continuous functions on the 
 Cantor space.  
 \qed
 \end{corollary}

The  algebras of the form $B(G)$ associated  with  uncountable graphs have other interesting 
properties. 
For example, 
it is not difficult to show that under the assumptions of Lemma~\ref{L.directed} and using the notation 
from its proof
the masas generated by $\{u_i: i\in Y\}$ and $\{v_x: x\in \bbA\}$
have the extension property. (Recall that a \emph{masa} in a C*-algebra is its  
maximal abelian C*-subalgebra and that it has the \emph{extension
property} if each of its pure states has the unique extension to a state of the algebra.)
This assertion is an immediate consequence of the following lemma. 

\begin{lemma} \label{L.uv} 
Assume $A$ is a C*-algebra and  $\phi$ is a state on $A$. 
Also assume $u$ is a  self-adjoint unitary in $A$ and $b\in A$ is  such that $ub=-bu$. 
If $\phi(u)=1$ or $\phi(u)=-1$  then $\phi(b)=0$. 
\end{lemma} 

\begin{proof}  Assume for a moment that  $\phi(u)=1$. 
The projection $p=(1+u)/2$ satisfies $\phi(p)=1$ and 
$pbu=-pub=-p(2p-1)b=-pb$ hence 
$pb(u+1)=0$ and $pbp=0$. But  the Cauchy--Schwartz inequality for $\phi$
easily implies  $\phi(b)=\phi(pbp)=0$  (see 
e.g.,  \cite[Lemma~3.5]{FaWo:Set}). 
The case when $\phi(u)=-1$ is analogous. 
\end{proof}

\section{Concluding remarks} 

It would be interesting to investigate  algebras $B(G)$ associated with 
uncountable graphs with strong partition properties (see~\cite{Moo:Solution}).

For every $n$ the algebra $M_n(\bbC)$ is the universal algebra generated by unitaries $u$ and $v$
such that $u^n=v^n=1$ and $uv=\gamma vu$, where $\gamma$ is a primitive $n$-th root of unity. 
Using this observation one can generalize algebras $B(G)$ by associating an AF algebra 
to a  digraph with labelled edges. At present I am not aware of  any applications of these algebras.

A positive answer to the Kishimoto--Ozawa--Sakai question 
would have rather interesting consequences. In
\cite{AkeWe:Consistency} it was proved that if Jensen's diamond
principle holds on $\aleph_1$ then  there is a counterexample to
Naimark's problem. 
If the results of \cite{KiOzSa} and \cite{FuKaKi}
extended to nonseparable nuclear C*-algebras then the argument from
\cite{AkeWe:Consistency} would show that Jensen's diamond principle
on any cardinal implies the existence of a counterexample to
Naimark's problem. It is not known whether positive answer to Naimark's problem is 
consistent with the standard axioms of set theory. 

\begin{question} \label{Q1} Assume $A$ is simple C*-algebra and $\phi$ and $\psi$ are its pure states. 
Is there a C*-algebra $B$ that has $A$ as a subalgebra and such that 
\begin{enumerate}
\item both $\phi$ and $\psi$ have unique state extensions, $\tilde\phi$ and $\tilde\psi$, 
to $B$, 
\item these extensions are equivalent, i.e., there is an automorphism
$\alpha$ of $B$ such that   $\tilde\phi=\tilde\psi\circ \alpha$. 
\end{enumerate}
\end{question} 

If the answer to Question~\ref{Q1} is positive, or if, for example,  for every simple nuclear algebra
$A$ one can find a simple nuclear algebra $B$ satisfying its requirements, then the argument from 
\cite{AkeWe:Consistency} shows that Jensen's diamond principle on any cardinal implies the existence
of a counterexample to Naimark's problem.

The following was suggested by Todor Tsankov. One can clearly ask a number 
of questions along these lines. 

\begin{question} 
Consider the algebra $B(G)$ constructed in 
the proof of Theorem~\ref{T3}. Are all of its irreducible representations on a separable Hilbert space
equivalent? 
\end{question}

Consider $\Gamma=G^{<\infty}$ (see the paragraph before Lemma~\ref{L.iso}) 
with respect to the symmetric 
difference $\Delta$ as a discrete  abelian group. The map 
$b\colon \Gamma^2\to \{-1,1\}$ defined by 
\[
b(\bfs,\bft)=(-1)^{|\{(x,y)\in \bfs\times \bft\colon x\text{ is adjacent to }y\}|}
\]
 satisfies  relations
 $b(\bfs,\bft)=1/b(\bfs,\bft)$
 and
 $b(\bfs,\bft_1) b(\bfs,\bft_2)=b(\bfs,\bft_1\bft_2)$.

 In \cite{Slaw:Factor}, J. Slawny associates a universal C*-algebra to a group $\Gamma$ 
and a function $b\colon \Gamma^2\to {\mathbb T}$ as above. 
For a Boolean group $\Gamma$, 
the CCR algebra associated to the pair 
$\Gamma,b$  
is always isomorphic to a group of the form $B(G)$. (Note that  Slawny considered
only second countable groups,  while in the present paper we consider uncountable discrete
groups.)
Consider $\Gamma$ as a vector space over ${\mathbb F}_2$ and 
fix its basis $V$. Proclaim two elements $x,y$ of $V$ to be adjacent if and only 
if $b(x,y)=-1$.  Then Lemma~\ref{L.iso} shows that $B(G)$ is isomorphic 
to Slawny's algebra. 
Among other things, Slawny proved that 
this algebra is simple if and only if the cocycle $c$ is nontrivial. 
This is related to the implication from \eqref{L.c.4} to \eqref{L.c.3} 
of Lemma~\ref{L.characterization}. 

All algebras of the form $B(G)$ are clearly AF, and 
as proved in Lemma~\ref{L.characterization} every simple algebra of the form 
$B(G)$ is a direct limit of full matrix algebras $M_{2^n}(\bbC)$ for $n\in \bbN$. 
While every separable algebra of this form is UHF, i.e., isomorphic to a tensor product of full matrix algebras,  in \cite{FaKa:Nonseparable} 
the author and T. Katsura 
proved that this may fail for nonseparable C*-algebras. 
For example, if 
 $\cZ$ is the Jiang--Su algebra (\cite{JiSu:On}) then  
  the algebra (below $\aleph_1$ denotes the least uncountable cardinal)
 \[
\textstyle  A_{\aleph_0\aleph_1}:=M_{2^\infty}\otimes \bigotimes_{\aleph_1} \cZ
 \]
  is a direct limit of full matrix algebras but not  UHF 
 (\cite[Proposition~3.2 and Theorem 1.3]{FaKa:Nonseparable}). 
 The algebra constructed in the proof of Theorem~\ref{T3} gives an another 
 example of an AM C*-algebra that is not UHF. This is because a nonseparable
 UHF algebra cannot
 have a representation on a separable Hilbert space (\cite[Proposition~7.6]{FaKa:Nonseparable}). 

\begin{conjecture} \label{C.1} 
The algebra $A_{\aleph_0\aleph_1}$ is not isomorphic to $B(G)$ 
for any graph~$G$. 
\end{conjecture} 

Clearly, not every AM algebra is isomorphic to an  algebra of the form $B(G)$---take, for example, 
$M_3(\bbC)$. (This is an immediate consequence
of Elliott's classification of AF algebras, see e.g., \cite{Ror:Classification}.)
 However, there is no K-theoretic obstruction to having a graph $G$
such that $A_{\aleph_0\aleph_1}$ is isomorphic to $B(G)$, since the K-theory of all these algebras
coincides with the K-theory of the CAR algebra. Thus a confirmation of Conjecture~\ref{C.1} 
would essentially confirm that the AM algebras that are also  CCR algebras form a  
nontrivial intermediate class between AM algebras and UHF algebras.

% This algebra has the property that the relative commutant of the copy of $M_{2^\infty}$ is 
 %isomorphic to $\bigotimes_{\aleph_1} \cZ$, and therefore has no nontrivial projections. 
% On the other hand, it is not difficult to show, by a proof similar to that of \cite[Lemma~4.3]{FaKa:Nonseparable}, that  every subalgebra of $B(G)$ of the form $B(G_0)$ for an induced
% subgraph $G_0$ of $G$ has relative commutant isomorphic to $B(H)$ 
 %for some graph $H\subseteq G^{<\infty}$. 
 %This implies that every separable subalgebra of  $B(G)$ is included in a separable 
 %subalgebra of $B(G)$ whose relative commutant is generated by its projections, 
 %a property not shared by 
% $M_{2^\infty}\otimes \bigotimes_{\aleph_1} \cZ$. 

Recall that for two graphs $G$ and $K$ 
we write $ G\sim K$ if C*-algebras $B(G)$ and $B(K)$ are isomorphic (see Lemma~\ref{L.iso}). 
The proof of Lemma~\ref{L.graphs} gives an algorithm that associates a natural number $k=k(G)$ 
to every finite graph $G$ such that $G\sim K$ if and only if $|V(G)|=|V(K)|$ and $k(G)=k(K)$. 

\begin{question} 
\label{Q.Complexity} 
What is the computational complexity of the relation $G\sim K$ for finite graphs $G$ and $K$? 
\end{question} 

Shortly after seeing a preliminary version of this paper, 
A. Kishimoto sketched a proof of Theorem~\ref{T3} using crossed products (\cite{Kish:Email}).

A word on precursors of the class of algebras considered here is in order. 
A variant of algebras of the form $B(G(Y,\bbA))$ with uncountable $Y$ 
was  used in~\cite{FaKa:Nonseparable} to answer a question of Jacques Dixmier.
After my presentation of \cite{FaKa:Nonseparable}  at the    COSy in Toronto in  
 May 2008, Bruce Blackadar  suggested what is essentially 
 a variant of $B(\bbN,\bbA)$ with an uncountable $\bbA$. 
  This example was reproduced   in  \cite[\S 7]{FaKa:Nonseparable} to give a partial answer to 
  a question of Masamichi Takesaki. 
 In all algebras used in \cite{FaKa:Nonseparable}  the family $\bbA$ includes all singletons 
of~$Y$.

\subsection{Acknowledgments} 
Some of the results of this note were first presented 
at the Set Theory workshop at the Erwin Schr\"odinger Institute and at the 11th Asian Logic Colloquium
in honour of Chitat Chong's 60th birthday in July 2009. The remaining results 
were proved during these two exciting meetings. I would like to thank the organizers of 
both meetings for inviting me. I would also like to thank George Elliott for pointing out to the connection
with \cite{Slaw:Factor}.  
Last, but not least, I would like to thank Takeshi Katsura 
for teaching me some of the techniques used here.
\providecommand{\bysame}{\leavevmode\hbox to3em{\hrulefill}\thinspace}
\providecommand{\MR}{\relax\ifhmode\unskip\space\fi MR }
\providecommand{\MRhref}[2]{%
  \href{http://www.ams.org/mathscinet-getitem?mr=#1}{#2}
}
\providecommand{\href}[2]{#2}

%\bibliographystyle{amsplain}
%\bibliography{nonhomogeneous}

\begin{thebibliography}{10}

\bibitem{AkeWe:Consistency}
C.~Akemann and N.~Weaver, \emph{Consistency of a counterexample to {N}aimark's
  problem}, Proc. Natl. Acad. Sci. USA \textbf{101} (2004), no.~20, 7522--7525.

\bibitem{Black:Operator}
B.~Blackadar, \emph{Operator algebras}, Encyclopaedia of Mathematical Sciences,
  vol. 122, Springer-Verlag, Berlin, 2006, Theory of $C\sp *$-algebras and von
  Neumann algebras, Operator Algebras and Non-commutative Geometry, III.

\bibitem{FaKa:Nonseparable}
I.~Farah and T.~Katsura, \emph{Nonseparable {UHF} algebras {I}: {D}ixmier's
  problem}, preprint, arXiv:0906.1401, 2009.

\bibitem{FaWo:Set}
I.~Farah and E.~Wofsey, \emph{Set theory and operator algebras}, notes from the
  Appalachian set theory workshop.
  http://www.math.cmu.edu/$\sim$eschimme/Appalachian/Index.html, 2008.

\bibitem{FuKaKi}
H.~Futamura, N.~Kataoka, and A.~Kishimoto, \emph{Homogeneity of the pure state
  space for separable {$C\sp \ast$}-algebras}, Internat. J. Math. \textbf{12}
  (2001), no.~7, 813--845.

\bibitem{JiSu:On}
X.~Jiang and H.~Su, \emph{On a simple unital projectionless {$C\sp
  *$}-algebra}, Amer. J. Math. \textbf{121} (1999), no.~2, 359--413.

\bibitem{Kish:Email}
A.~Kishimoto, an email to {S}. {S}akai, July 02, 2009.

\bibitem{KiOzSa}
A.~Kishimoto, N.~Ozawa, and S.~Sakai, \emph{Homogeneity of the pure state space
  of a separable {$C\sp *$}-algebra}, Canad. Math. Bull. \textbf{46} (2003),
  no.~3, 365--372.

\bibitem{Ku:Book}
K.~Kunen, \emph{An introduction to independence proofs}, North--Holland, 1980.

\bibitem{Moo:Solution}
J.~T. Moore, \emph{A solution to the {$L$} space problem}, J. Amer. Math. Soc.
  \textbf{19} (2006), no.~3, 717--736.

\bibitem{Ror:Classification}
M.~R{\o}rdam, \emph{Classification of nuclear {$C\sp *$}-algebras},
  Encyclopaedia of Math. Sciences, vol. 126, Springer-Verlag, Berlin, 2002.

\bibitem{Slaw:Factor}
J.~Slawny, \emph{On factor representations and the {$C\sp{\ast} $}-algebra of
  canonical commutation relations}, Comm. Math. Phys. \textbf{24} (1972),
  151--170.

\bibitem{We:Set}
N.~Weaver, \emph{Set theory and {$C^*$}-algebras}, Bull. Symb. Logic
  \textbf{13} (2007), 1--20.

\end{thebibliography}
\end{document}